\documentclass[12pt,reqno,intlim]{amsart}

\usepackage{amssymb,amsmath}

\usepackage[arrow,matrix,curve]{xy}

\newdir{ >}{{}*!/-5pt/\dir{>}}

\newcommand{\nc}{\newcommand}

\nc{\bC}{\bold{C}} \nc{\bN}{\Bbb{N}} \nc{\cF}{\mathcal{F}}
\nc{\cE}{\mathcal{E}} \nc{\cR}{\mathcal{R}} \nc{\cM}{\mathcal{M}}
\nc{\al}{\alpha} \nc{\bt}{\beta} \nc{\gm}{\gamma} \nc{\dl}{\delta}
\nc{\om}{\omega} \nc{\sg}{\sigma} \nc{\Sg}{\Sigma} \nc{\vf}{\varphi}
\nc{\ve}{\varepsilon} \nc{\os}{\overset} \nc{\ol}{\overline}
\nc{\ul}{\underline} \nc{\us}{\underset} \nc{\sbs}{\subset}
\nc{\bsl}{\backslash} \nc{\Ra}{\Rightarrow}
\nc{\lra}{\longrightarrow} \nc{\all}{\allowdisplaybreaks}

\nc{\Codes}{\operatorname{{\bold{Codes}}}}
\nc{\RegMono}{\operatorname{\mathcal{R}{\rm{eg}\mathcal{M}{\rm{ono}\!}}}}
\nc{\RegEpi}{\operatorname{\mathcal{R}{\rm{eg}\mathcal{E}{\rm{pi}\!}}}}
\nc{\Mn}{\operatorname{\mathcal{M}{\rm{ono}\!}}}
\nc{\Ep}{\operatorname{\mathcal{E}{\rm{pi}\!}}}
\nc{\Rg}{\operatorname{\mathcal{R}{\rm{eg}\!}}}
\nc{\Ob}{\operatorname{Ob\!}}

\numberwithin{equation}{section}

\newtheorem{theo}{\ \ \ Theorem}[section]
\newtheorem{lem}[theo]{\ \ \ Lemma}
\newtheorem{prop}[theo]{\ \ \ Proposition}
\newtheorem{cor}[theo]{\ \ \ Corollary}

\theoremstyle{definition}
\newtheorem{exmp}[theo]{\ \ \ Example}

\theoremstyle{remark}

\begin{document}

\title[]
{Cokernels in the stable category of a left hereditary ring}

\author{Dali Zangurashvili}

\maketitle

\begin{abstract}
It is proved that if a ring is left hereditary, left perfect and right coherent, then the stable category has cokernels. Moreover, we show that the condition for a ring to be left perfect and right coherent is also necessary for the stable category to have cokernels, provided that the ring is left hereditary and satisfies the additional condition that there are no non-trivial projective injective left modules over it (satisfied, for instance, by integral domains). This, in particular, implies that, for a Dedekind domain, the stable category has cokernels if and only if the domain is left perfect. Several new necessary and sufficient conditions for a left hereditary ring to be left perfect and right coherent are found. One of them requires that the full subcategory of projective modules be reflective in the category of modules. Another one requires that any module be isomorphic to a stable module in the stable category. Yet another equivalent condition found in the paper requires that, for any module $M$, among all representations $M=K\oplus P$ with a projective $P$, there should be the one with the smallest $K$. To accomplish the goals, a version of the well-known Freyd's adjoint functor theorem, where the solution set condition is removed under some additional conditions on the categories, is given. 

\noindent{\bf Key words and phrases}: stable category, cokernel, left hereditary ring, left perfect ting, right coherent, stable module.

\noindent{\bf 2020  Mathematics Subject Classification}: 18G65, 18A30, 16E60, 18A40.
\end{abstract}

\section{Introduction}
In our joint paper \cite{MZ} with Alex Martsinkovsky we studied some properties of the stable category $\underline{\Lambda-Mod}$ of a left hereditary ring $\Lambda$. In particular, we proved that this category has kernels and finite products, and if, in addition, a ring is left perfect and right coherent, then the stable category is complete. We also showed that the stable category has coproducts (for, not only left hereditary, but an arbitrary ring). However, the existence of cokernels was proved only in the case where the ring $\Lambda$ is the finite direct product of complete blocked triangular matrix algebras over a division ring (note that the rings $\Lambda$ of the described kind are precisely rings for which the stable category is Abelian, as it was proved in \cite{MZ}).

In the present paper we generalize the above-mentioned result on cokernels in the stable category. Namely, we prove that if a ring is left hereditary, left perfect and right coherent, then the stable category has cokernels. Moreover, we show that  the condition for a ring to be left perfect and right coherent is also necessary for the stable category to have cokernels, provided that the ring is left hereditary and satisfies the additional condition that there are no non-trivial projective injective left modules over it (satisfied, for instance, by integral domains). This, in particular, implies that, for a Dedekind domain, the stable category has cokernels if and only if the domain is left perfect. 

There is the well-known criterion for a ring to be left perfect and right coherent due to Chase \cite{C}. It asserts that a ring is such if and only if direct product of any family of projective modules is projective. In our joint paper \cite{MZ} we gave another criterion for a ring to be of this kind, and also two more criteria for a ring to be left hereditary ring, left perfect and right coherent formulated in terms of the quotient functor $Q:\Lambda-Mod\rightarrow \underline{\Lambda-Mod}$. 

 In this paper we find several new criteria for a ring to be of the latter kind. One of them asserts that a ring is left hereditary, left perfect and right coherent if and only if the full subcategory of projective modules is epireflective in the category of all left modules. Here "epireflective" can be replaced by "reflective" provided that a ring is left hereditary. Another equivalent condition found in this paper, for rings of the latter kind, requires that any module be isomorphic to a stable module in the stable category. Yet another equivalent condition requires that, for any module $M$, among all representations $M=K\oplus P$ with a projective $P$, there should be the one with the smallest $K$ (which obviously is stable).

To accomplish our goals we give a version of the well-known Freyd's adjoint functor theorem, where the solution set condition is removed under some additional conditions on the categories.

The author expresses her gratitude to Alex Martsinkovsky and other participants of the seminar held in Tbilisi, 2019, under the framework of the research project "Stable Structures in Homological Algebra" for valuable talks on the subject of the paper. 

Financial support from  Shota Rustaveli  Georgian National Science Foundation
(Ref.: FR-18-10849) is gratefully acknowledged.

\section{Stable modules}
Throughout the paper $\Lambda$ denotes an associative ring with identity and $\Lambda$-$Mod$ denotes the category of left $\Lambda$-modules. Moreover, "module" means a left $\Lambda$-module.

Recall that a module is called \textit{stable} if it has no non-zero projective summands.
\begin{lem} For the following conditions, we have 
\begin{center}
(i)$\Leftrightarrow$(ii)$\Leftarrow$(iii)$\Leftarrow$(iv). 
\end{center}
If $\Lambda$ is left hereditary, then the conditions (i)-(iv) are equivalent:\vskip+2mm

(i) A module $M$ is stable;\vskip+1mm

(ii) there is no an epimorphism $M\twoheadrightarrow P$ with projective non-trivial $P$;\vskip+1mm

(iii) there is no non-trivial homomorphism $M\rightarrow \Lambda$; \vskip+1mm

(iv) there is no non-trivial homomorphism $M\rightarrow P$ with projective $P$.
\vskip+1mm
\end{lem}

\begin{proof}
The implication (iii)$\Rightarrow$(i) and its converse (in the case of a left hereditary ring), is shown in \cite{MZ} (see Lemma 2.6). All other implications are obvious.
\end{proof}

\begin{cor}
If $M\twoheadrightarrow N$ is an epimorphism and $M$ is stable, then $N$ is stable.
\end{cor}

Let $\Lambda$-$Mod_{St}$ be the full subcategory of $\Lambda$-$Mod$ with objects being stable left modules.

\begin{cor}
The subcategory $\Lambda$-$Mod_{St}$ is closed under cokernels.
\end{cor}

\begin{lem}
Let $\Lambda$ be left hereditary. The direct sum of arbitrary family of stable left $\Lambda$-modules is stable if and only if all these modules are stable.
\end{lem}

\begin{proof}
It suffices to observe that the existence of a non-zero homomorphism from the direct sum to $\Lambda$ is equivalent to the existence of such a homomorphism from at least one of summands, and then to apply Lemma 2.1.
\end{proof}

Corollary 2.3 and Lemma 2.4 immediately imply the following statements.

\begin{cor}
Let $\Lambda$ be left hereditary. The category $\Lambda$-$Mod_{St}$ is additive.
\end{cor}

\begin{lem}
Let $\Lambda$ be left hereditary. Then the subcategory $\Lambda$-$Mod_{St}$ is cocomplete and the embedding functor \begin{center}
$L:\Lambda$-$Mod_{St}\rightarrow \Lambda$-$Mod$
\end{center} \noindent preserves them.
\end{lem}

Before continue we give a version of the well-known Freyd's adjoint functor theorem.  We will use this statement three times in the paper. It follows from the version suggested by us in \cite{Z} (see Remark 3.1 (2)). But for the reader's convenience, we give here its direct proof. 

Let $\mathcal{C}$ be a category, and $\mathbb{E}$ be a morphism class. For any object $C$ of $\mathcal{C}$, an equivalence relation on the class of $\mathbb{E}$-morphisms with a domain $C$ arises in a natural way. Recall that $\mathcal{C}$ is said to be $\mathbb{E}$-co-well-powered if the quotient of the latter class by this equivalence relation is a set (see, e.g., \cite{B}). For the definition of a factorization system we refer the reader to the same book. Recall that if $(\mathbb{E},\mathbb{M})$ is a factorization system on $\mathcal{C}$, then the intersection $\mathbb{E}\cap\mathbb{M}$ coincides with the class of all isomorphisms of $\mathcal{C}$. Moreover, if $\mathbb{E}\subseteq Epi$, then the morphism class $\mathbb{M}$ satisfies the left cancellation property, i.e. if $gf\in \mathbb{M}$, then $f\in \mathbb{M}$ (see, e.g., \cite{Z1}).

Below by the symbols $Epi$ and $Mono$ we denote the classes of resp. epimorphisms and monomorphisms of the category $\mathcal{C}$. 

\begin{theo}
Let $\mathcal{C}$ be a complete category, and let $(\mathbb{E},\mathbb{M})$ be a factorization system on $\mathcal{C}$ with $\mathbb{E}\subseteq Epi$. Let $\mathcal{C}$ be $\mathbb{E}$-co-well-powered, and let $\mathcal{X}$ be a full replete subcategory of $\mathcal{C}$. The following conditions are equivalent:\vskip+2mm

(i) $\mathcal{X}$ is closed under products and $\mathbb{M}$-subobjects (the latter means that if a morphism $m:X\rightarrow X'$ lies in $\mathbb{M}$ and $X'$ is an object of $\mathcal{X}$, then $X$ also is an object of $\mathcal{X}$);\vskip+2mm 

(ii)  $\mathcal{X}$ is a reflective subcategory of $\mathcal{C}$, and the reflection unit components lie in $\mathbb{E}$;\vskip+2mm

(iii) $\mathcal{X}$ is a reflective subcategory of $\mathcal{C}$ and is closed under $\mathbb{M}$-subobjects.
\end{theo}

\begin{proof} (i)$\Rightarrow$(ii): Let $C$ be an object of $\mathcal{C}$. For any equivalence class $W$ of $\mathbb{E}$-morphisms with domain $C$ and codomains in $\mathcal{X}$, consider a representative $e_W$. The class of all morphisms $e_W$ is a set. Consider the product $A=\prod_{W}codom \ e_W$. Note that $A$ is in $\mathcal{X}$. The family of morphisms $e_W$ induces a morphism $\varrho:C\rightarrow A$. Consider the $(\mathbb{E},\mathbb{M})$-factorization $me$ of $\varrho$. Obviously the codomain $A'$ of the morphism $e'$ is an object of $\mathcal{X}$. We assert that $e$ is a universal arrow from $C$ to the subcategory $\mathcal{X}$. Indeed, consider any $f:C\rightarrow X'$ with $X'$ from $\mathcal{X}$. Let $m'e'$ be the  $(\mathbb{E},\mathbb{M})$-factorization of $f$ with  $e':C\twoheadrightarrow Y$. Then $Y$ also is an object of $\mathcal{X}$. Without loss of generality we can assume that the chosen representative of $e'$'s equivalence class is $e'$ itself. Therefore $\pi_{e'}\varrho=e'$, where $\pi_{e'}$ is the canonical projection $A\rightarrow codom \ e_W$. Then $(m'\pi_{e'} m)e=f$.


(ii)$\Rightarrow$(iii): Consider a morphism $m:X\rightarrow X'$ from $\mathbb{M}$ with $X'$ being an object of $\mathcal{X}$. Let $\eta$ be the unit of the reflection $r$. There is a morphism $h:r(X)\rightarrow X'$ such that $m=h\eta_X$. Hence $\eta_X\in \mathbb{M}$. This implies that $\eta_X$ is an isomorphism, and hence $X$ is an object of $\mathcal{X}$.

The implication (iii)$\Rightarrow$(i) is well-known.
\end{proof}

From Corollary 2.2, Lemma 2.4 and the dual of Theorem 2.7 we obtain

\begin{prop}
Let $\Lambda$ be left hereditary. The category $\Lambda$-$Mod_{St}$ is a coreflective subcategory of $\Lambda$-$Mod$. The coreflector $R$ sends a module $M$ to the sum of all its stable submodules, and the corresponding component of the counit $\varepsilon$ of the reflection is the embedding $R(M)\rightarrowtail M$.
\end{prop}



Proposition 2.8 implies
\begin{cor} Let $\Lambda$ be left hereditary. Then
\vskip+2mm
(a) the category $\Lambda$-$Mod_{St}$ of stable modules is complete;
 \vskip+2mm
(b) the kernel of a morphism $f:A\rightarrow B$ in $\Lambda$-$Mod_{St}$ is the sum of all stable submodules of $Ker f$;
\vskip+2mm
(c) the product in $\Lambda$-$Mod_{St}$ of finite number of stable modules is their direct sum. The product in $\Lambda$-$Mod_{St}$ of infinite number of stable modules is the sum of stable submodules of their product in $\Lambda$-$Mod$.
\end{cor}

Corollaries 2.5 and 2.9 imply

\begin{cor}
Let $\Lambda$ be left hereditary. A morphism $f:A\rightarrow B$ of $\Lambda$-$Mod_{St}$ is a monomorphism if and only if $Ker f$ does not contain a non-trivial stable submodule.
\end{cor}

From Proposition 2.8 we immediately obtain
\begin{cor}
Let $\Lambda$ be left hereditary. A morphism $A\rightarrow B$ in $\Lambda$-$Mod_{St}$ is an epimorphism if and only if it is an epimorphism in $\Lambda$-$Mod$.
\end{cor}


Dually to the notion of a stable module one can introduce the notion of a \textit{costable} module. The duals of all statements of this section hold. In particular, we have the following statement, it follows from Theorem 2.7.  
\begin{prop}
Let $\Lambda$ be a left hereditary ring. The full subcategory $\Lambda$-$Mod_{CoSt}$ of the category $\Lambda$-$Mod$ with objects being costable modules is reflective. The reflector $L'$ sends a module $M$ to the quotient $M/(\cap_{f:M\twoheadrightarrow C,\ C\ is\ costable} Ker f)$, and the components of the unit of the reflection are the canonical projections.
\end{prop}
Let \begin{center}$R':\Lambda$-$Mod_{CoSt}\rightarrow \Lambda$-$Mod$\end{center} be the embedding functor.
\begin{cor}
The functor \begin{center}
$L'L:\Lambda$-$Mod_{St}\rightarrow \Lambda$-$Mod_{CoSt}$
\end{center} is a left adjoint of the functor $RR'$.
\end{cor}
\vskip+5mm
\section{The stable category and the category of stable modules}
Let $\underline{\Lambda-Mod}$ be the stable category. Recall that its objects are modules, while morphisms $M\rightarrow N$ are equivalence classes of the set $Hom(M,N)$ with respect to the equivalence relation $\sim$ defined as follows: $f\sim g$ if the homomorphism $(f-g)$ factors through a projective. The equivalence class of a homomorphism $f$ is denoted by $\underline{f}$. The quotient functor $$\Lambda-Mod\rightarrow \underline{\Lambda-Mod}$$ is denoted by $Q$.

From now on, in this section $\Lambda$ \textit{is assumed to be a left hereditary ring}.
\begin{lem}(\cite{MZ}, Lemma 2.7)
Equivalent homomorphisms with a stable domain are equal.
\end{lem}

Consider the restriction $Q_{St}$ of the functor $Q$ on the subcategory $\Lambda$-$Mod_{St}$. 
Lemma 3.1 implies 
\begin{lem}
The functor $Q_{St}$ is full and faithful.
\end{lem}

Further, note that there is a functor $$R_{St}:\underline{\Lambda-Mod}\rightarrow \Lambda-Mod_{St}$$ such that the following triangle is commutative:
\begin{equation}
\xymatrix{\Lambda-Mod_{St}& \Lambda-Mod\ar[dl]^{Q}\ar[l]_{R}\\
\underline{\Lambda-Mod}\ar[u]^{R_{St}}}
\end{equation}
\vskip+2mm
\noindent The functor $R_{St}$ acts on objects in the same way as $R$ does. Its action on morphisms coincides with that of $R$ on their representatives. We only need to show that if $f,g:M\rightarrow N$ are arbitrary homomorphisms with $f\sim g$, then $f\mid_{R(M)}=g\mid_{R(M)}$. But this follows from the fact that $R(M)$ is stable and Lemma 3.1. 
\begin{prop}
The functor $R_{St}$ is a right adjoint to the functor $Q_{St}$.
\end{prop}
\begin{proof}
By Theorem 4.10 of \cite{MZ}, the functor $Q$ preserves monomorphisms. Applying this fact one can easily verify that $\underline{\varepsilon_M}$ is couniversal from the functor $Q_{St}$ to $M$, for any module $M$.
\end{proof}
\begin{cor}
If the functor $Q$ has a left adjoint, then the latter one is identical on stable modules (up to an isomorphism).
\end{cor}

\begin{proof}
If a left adjoint of $Q$ exists, then the composition of $Q_{St}$ with it is a left adjoint of $R$, and hence is isomorphic to $L$.
\end{proof}

\section{Some conditions on a ring}
Let $\Lambda$ be an arbitrary ring. For the main result of the paper, we need the following 
\begin{lem}
For the following conditions, we have the implications:
\begin{center}
(i)$\Leftarrow$(ii)$\Leftrightarrow$(iii)$\Leftarrow$(iv)$\Leftrightarrow$(v)$\Leftrightarrow$(vi)$\Leftrightarrow$(vii)$\Leftrightarrow$(viii).
\end{center}
\vskip+2mm
\noindent If $\Lambda$ is left hereditary, then the conditions (i)-(viii) are equivalent:
\vskip+2mm
(i) there is a module embedding $\Lambda\rightarrowtail S$ with a stable module $S$;
\vskip+2mm
(ii) any free module can be embedded into a stable module;
\vskip+2mm
(iii) any module can be embedded into a stable module;
\vskip+2mm
(iv) any injective module is stable;
\vskip+2mm
(v) there is no a non-trivial projective injective module;
\vskip+2mm
(vi) any projective module is costable;
\vskip+2mm
(vii) for any module $M$, there is an epimorphism $C\twoheadrightarrow M$ with a costable module $C$;
\vskip+2mm
(viii) for any free module $M$, there is an epimorphism $C\twoheadrightarrow M$ with a costable module $C$;
\vskip+2mm

\end{lem}
\begin{proof}
The equivalence (ii)$\Leftrightarrow$(iii) follows from Corollary 2.3.


(v)$\Rightarrow$(iv): Let $I$ be an injective module and $P$ be its projective direct summand. Obviously $P$ is injective, and hence is $\lbrace 0\rbrace$. The proof of the implication (v)$\Rightarrow$(vi) is similar.

The implication (vii)$\Rightarrow$(vi) follows from the fact that any submodule of a costable module is costable (which is dual of Corollary 2.2).

All other implications are obvious. 

Let now $\Lambda$ be left hereditary. The implication (i)$\Rightarrow$(ii) follows from Lemma 2.4. 

(iii)$\Rightarrow$(iv): Let $I$ be an injective module and $I\rightarrowtail S$ be a monomorphism with a stable module $S$. Then $I$ is a direct summand of $S$. By Lemma 2.4 it is stable.\vskip+1mm

\end{proof}

\begin{exmp}
(a) Any integral domain that is not a field satisfies the condition (v) of Lemma 4.1 \cite{R}.


(b) A primary algebra over a field with radical square zero and having vector space dimension greater than two satisfies the condition (v) of Lemma 4.1 \cite{J}.

(c) A left hereditary ring $\Lambda$ such that the injective envelope of $_{\Lambda}\Lambda$, considered as a left module over itself, is projective obviously does not satisfy the condition (v) of Lemma 4.1. Such rings were characterized by R. R Colby and E. A. Rutter in \cite{CR}. In \cite{MZ} we proved that a left hereditary ring is of this kind if and only if the stable category of this ring is Abelian.
\end{exmp}
 \section{The main theorem}
 \begin{lem}
Let $\Lambda$ be left hereditary, and $M$ be a module. The conditions (i)-(iii) below are equivalent and imply the condition (iv).  If $f:M\rightarrowtail N$ is a monomorphism with a stable codomain $N$, then the conditions (i)-(v) are equivalent:
\vskip+2mm
(i) there is an epimorphism $\alpha:M\twoheadrightarrow P$ which is universal from $M$ to the full subcategory of projective modules (i.e., $\alpha$ is such that for any $\alpha':M\rightarrow P'$ with projective $P'$, there is (a unique) $\beta:P\rightarrow P'$ such that $\beta\alpha=\alpha'$);
\vskip+2mm
(ii) there is a representation $M=K\oplus P$ with submodules $K$ and $P$ such that $P$ is projective and if $M=K'\oplus P'$ for submodules $K'$ and $P'$ of $M$ with projective $P'$, then $K\subseteq K'$ and there is a projective submodule $P''$ of $P$  such that $P=P'\oplus P''$ and $K'=K\oplus P''$;
\vskip+2mm
(iii) there is a representation $M=K\oplus P$ with submodules $K$ and $P$ such that $P$ is projective and  if $M=K'\oplus P'$ for submodules $K'$ and $P'$ of $M$ with projective $P'$, then $K\subseteq K'$;\vskip+2mm
(iv) the morphism $\underline{g}$ of the stable category has a cokernel, for any homomorphism $g:M\rightarrow N'$ with a stable codomain $N'$. \vskip+2mm
(v) the morphism $\underline{f}$ of the stable category has a cokernel.
\vskip+2mm
If the conditions (i)-(iii) are satisfied, then the cokernel of $\underline{f}$ is the morphism $\underline{\sigma}$, where $\sigma$ is the pushout of the canonical projection $\pi:K\oplus P\twoheadrightarrow P$ of the representation from the condition (iii) (which is a universal arrow from the module $M=K\oplus P$ to the subcategory of projective modules), along the monomorphism $f$:
\begin{equation}
\xymatrix{K\oplus P\ar@{ >->}[r]^{f}\ar@{->>}[d]_{\pi}&N\ar@{->>}[d]^{\sigma}\\
P\ar@{ >->}[r]&D
}
\end{equation}\vskip+2mm

\end{lem}
\begin{proof}
(i)$\Rightarrow$(ii): Consider the representation $M=K\oplus P$ induced by the morphism $\alpha$. For a representation $M=K'\oplus P'$, let $\alpha':M=K'\oplus P'\twoheadrightarrow P'$ be the canonical projection. There is a homomorphism $\beta:P\rightarrow P'$ with $\beta\alpha=\alpha'$. This implies that $K$ is a submodule of $K'$. Moreover, the homomorphism $\beta$ is an epimorphism, and hence $P=P'\oplus P''$ for the submodule $P''=Ker\ \beta$, which is projective. Let $i$ be the canonical embedding $P\rightarrowtail K\oplus P$. Then $\alpha'i=\beta$. Therefore $Ker\ \beta$ is a submodule of $K'$. Since we have $K\oplus P''\oplus P'=K'\oplus P'$, and $K\oplus P''$ is a submodule of $K'$, we conclude that $K'=K\oplus P''$.

(ii)$\Rightarrow$(iii) is obvious.

(iii)$\Rightarrow$(i) Let $\alpha$ be the canonical projection $K\oplus P\twoheadrightarrow P$. Consider an arbitrary homomorphism $\alpha':M\rightarrow P'$ with projective $P'$. Since the ring is left hereditary, $\alpha'(M)$ is projective, and hence $M=Ker\ \alpha'\oplus \alpha'(M)$. Therefore $K\subseteq Ker\ \alpha'$. This implies the existence of a homomorphism $\beta$ with $\beta\alpha=\alpha'$.\vskip+1mm

(iii)$\Rightarrow$(iv): We will prove the existence of a cokernel for the morphism $\underline{f}$. But the arguments obviously work for an arbitrary $\underline{g}$ with a stable codomain. Consider pushout (5.1). We will show that $\underline{\sigma}$ is a cokernel of $\underline{f}$. Assume $\underline{\sigma' f}=0$, for some $\sigma':N\rightarrow D'$. Then there is a commutative diagram 
\begin{equation}
\xymatrix{&M\ar@{ >->}[r]^{f}\ar@{->>}[d]_{\alpha}\ar[ddl]_{\alpha'}&N\ar[d]^{\sigma}\ar[ddr]^{\sigma'}\\
&P\ar[r]^{\beta}\ar@{-->}[dl]^{\delta}&D\ar@{-->}[dr]_{\varphi}\\
P'\ar[rrr]^{\beta'}&&&D'}
\end{equation}
for some projective $P'$ and some homomorphisms $\alpha'$, $\beta'$, $\delta$ and $\varphi$. The existence of $\delta$ follows from the fact that the image of $M$ under $\alpha'$ is projective, and hence is a direct summand of $M$, while the existence of $\varphi$ follows from the fact that the internal square of diagram (5.2) is a pushout. If $\varphi':D\rightarrow D'$ is such that $\underline{\varphi'\sigma}=\underline{\sigma'}$, then, since $N$ is stable, $\varphi'\sigma=\sigma'$. Since $\sigma$ is an epimorphism, we can conclude that $\varphi'=\varphi$. 

The implication (iv)$\Rightarrow$(v) is obvious.

(v)$\Rightarrow$(i): Let $\underline{\sigma}$ be a cokernel of $\underline{f}$ (in the stable category), for some $\sigma:N\rightarrow D$. Since $\underline{f\sigma}=0$, there is a commutative square 
\begin{equation}
\xymatrix{M\ar@{ >->}[r]^{f}\ar@{->>}[d]_{\alpha}&N\ar[d]^{\sigma}\\
P\ar[r]&D
}
\end{equation}
\noindent  with  projective $P$ and an epimorphism $\alpha$. 

Let $\alpha':M\rightarrow P'$ be a homomorphism with projective $P'$. Consider the pushout $\sigma'$ of $\alpha'$ along $f$ (see diagram (5.2)). 
Obviously $\underline{f\sigma'}=0$. Hence there is a morphism $\varphi:D\rightarrow D'$ such that $\underline{\sigma'}=\underline{\varphi\sigma}$. Since $N$ is stable, this implies that $\sigma'=\varphi\sigma$. The homomorphism $\beta'$, being a pushout of a monomorphism (in the Abelian category $\Lambda$-$Mod$), is a monomorphism. Since the pair of morphism classes $(Epi, Mono)$  is a factorization system in the category of $\Lambda$-modules, there is a morphism $\delta:P\rightarrow P'$ making diagram (5.2) commutative. This, together with the fact that $\alpha$ is an epimorphism, implies that it is universal from $M$ to the subcategory of projective modules.\end{proof}

\begin{theo}
For a ring $\Lambda$, the following conditions are equivalent:
\vskip+2mm

(i) the ring $\Lambda$ is left hereditary, left perfect and right coherent;\vskip+2mm

(ii) the ring $\Lambda$ is left hereditary and the direct product of arbitrary family of projective modules is projective;\vskip+2mm

(iii) the functor $Q$ preserves limits;\vskip+2mm

(iv) the functor $Q$ has a left adjoint;\vskip+2mm

(v) the full subcategory of projective modules of $\Lambda$-$Mod$ is epireflective;
\vskip+2mm 

(vi) the ring $\Lambda$ is left hereditary and the full subcategory of projective modules of $\Lambda$-$Mod$ is reflective;
\vskip+2mm 

(vii) the ring $\Lambda$ is left hereditary and for any module $M$, there is a representation $M=K\oplus P$ with submodules $K$ and $P$ such that $P$ is projective and if $M=K'\oplus P'$ for submodules $K'$ and $P'$ of $M$ with projective $P'$, then $K\subseteq K'$ and there is a projective submodule $P''$ of $P$  such that $P=P'\oplus P''$ and $K'=K\oplus P''$;
\vskip+2mm

(viii) the ring $\Lambda$ is left hereditary and for any module $M$, there is a representation $M=K\oplus P$ with submodules $K$ and $P$ such that $P$ is projective and if $M=K'\oplus P'$ for submodules $K'$ and $P'$ of $M$ with projective $P'$, then $K\subseteq K'$;
\vskip+2mm

(ix) the ring $\Lambda$ is left hereditary and any module is isomorphic to some stable module in the stable category;\vskip+2mm

(x) the ring $\Lambda$ is left hereditary and the functors $R_{St}$ and $Q_{St}$ are equivalences of categories;\vskip+2mm

(xi) the ring $\Lambda$ is left hereditary and for any module $M$, the morphism $\underline{\varepsilon_M}:R(M)\rightarrowtail M$ of the stable category is an isomorphism;\vskip+2mm

(xii) the ring $\Lambda$ is left hereditary and for any module $M$, there exists the representation $M=N\oplus P$ such that $P$ is  projective and $N$ is a submodule of $R(M)$;\vskip+2mm

(xiii) the ring $\Lambda$ is left hereditary and for any module $M$, there are projective modules $P$ and $P'$ and an isomorphism $\delta:R(M)\oplus P\rightarrow M\oplus P'$ such that $\varepsilon_M=\pi'\delta i$, where $i$ and $\pi'$ are resp. the canonical embedding and the canonical projection:
\begin{equation}
\xymatrix{R(M)\oplus P\ar[r]_{\delta}^{\approx}&M\oplus P'\ar@{->>}[d]^{\pi'}\\
R(M)\ar@{ >->}[r]_{\varepsilon_M}\ar@{ >->}[u]^{i}&M}
\end{equation} 
\vskip+2mm

\end{theo}
\begin{proof} The equivalence of the conditions (i)$\Leftrightarrow$(ii) is the well-known result by Chase \cite{C}. The equivalence of the conditions (ii)-(iv) is proved in \cite{MZ}. \vskip+1mm

The equivalences (ii)$\Leftrightarrow$(v)$\Leftrightarrow$(vi) follow from Theorem 2.7 (take $\mathbb{E}=Epi$ and $\mathbb{M}=Mono$). 

The equivalences (v)$\Leftrightarrow$(vii)$\Leftrightarrow$(viii) follow Lemma 5.1 (since we already know that the condition (v) implies that the ring $\Lambda$ is left hereditary). \vskip+1mm

The implication (viii)$\Rightarrow$(ix) follows from the fact that the module $K$ in the condition (viii) is stable.\vskip+1mm

The implication (ix)$\Rightarrow$(x) follows from Lemma 3.2 and Proposition 3.3.\vskip+1mm

The implication (x)$\Rightarrow$(iv) follows from the commutativity of diagram (3.1).\vskip+1mm

The equivalence (x)$\Leftrightarrow$(xi) follows from Lemma 3.2 and the fact that $\underline{\varepsilon}$ is the counit of the adjunction $Q_{St}\dashv R_{St}$.

The equivalence of the conditions (xi) and (xii) follows from Proposition 5.12 of \cite{MZ}, while the equivalence of (xi) and (xiii) follows from the well-known Heller's criterion \cite{H}.\vskip+1mm

\end{proof}

In \cite{MZ} we showed that if a ring is left hereditary, left perfect and right coherent, then the stable category is complete. We also showed that the stable category has arbitrary coproducts (for an arbitrary ring; see Proposition 3.5 of \cite{MZ}). Theorem 5.2 implies
\begin{theo}
Let $\Lambda$ be left hereditary. For the following conditions we have (i)$\Rightarrow$(ii)$\Leftrightarrow$(iii). Under the equivalent conditions of Lemma 4.1, all three conditions below are equivalent:\vskip+2mm

(i) the ring $\Lambda$ is left perfect and right coherent;\vskip+2mm

(ii) the stable category of the ring $\Lambda$ has cokernels;\vskip+2mm

(iii) the stable category of the ring $\Lambda$ is cocomplete.
\end{theo}
\begin{proof}
(i)$\Rightarrow$(ii): The condition (x) of Theorem 5.2 and Lemma 2.6 imply the claim. 

(ii)$\Rightarrow$(i): Lemma 5.1 implies the condition (v) of Theorem 5.2.\vskip+1mm
\end{proof}
Note that Theorem 5.3 generalizes Proposition 9.4 of \cite{MZ} which asserts that the stable category is cocomplete if $\Lambda$ is left hereditary and the injective envelope of $_{\Lambda}\Lambda$ is projective. By Proposition 9.3 of \cite{MZ}, rings of this kind are left perfect and right coherent.
\vskip+1mm
Theorem 5.3 and Example 4.2 (a) imply
\begin{cor}
Let $\Lambda$ be a Dedekind domain. The stable category of $\Lambda$ has cokernels if and only if $\Lambda$ is left perfect.
\end{cor}

\begin{proof}
It suffices to note that any Dedekind domain is Noetherian and hence right coherent.
\end{proof}


It is easy to observe that for a left hereditary, left perfect and right coherent ring, the cokernel of a morphism $\underline{f}$, for $f:K\rightarrow M$, is the pair $(R(M)/R(K), coker R(f)m\pi)$, where $\pi$ is the canonical projection of the representation from the condition (xii), while $m$ is the embedding $N\rightarrowtail R(M)$ from the same condition. However, in fact, there is much simpler construction of a cokernel in the stable category (provided that a ring is left hereditary, left perfect and right coherent). To give it, first recall one of several criteria for an epimorphism of modules to be an epimorphism in the stable category given in \cite{MZ}.
\begin{theo} (\cite{MZ}, Theorem 5.7)
Let $\pi:M\twoheadrightarrow N$ be an epimorphism of modules. The following conditions are equivalent:
\vskip+2mm
(i) the morphism $\underline{\pi}$ of the stable category is an epimorphism;
\vskip+2mm
(ii) for any projective module $P$ and a homomorphism $g:M\rightarrow P$ there is a morphism $s:N\rightarrow M$ such that $\pi's=g'$, where the following square is a pushout in $\Lambda$-$Mod$:
\begin{equation}
\xymatrix{M\ar[d]_{g}\ar@{->>}[r]^{\pi}&N\ar[d]^{g'}\ar@{-->}[dl]^{s}\\
P\ar@{->>}[r]_{\pi'}&D}
\end{equation}
\end{theo}
\vskip+3mm
\begin{prop}
Let $\Lambda$ be hereditary, and let $f:S\rightarrow M$ be a homomorphism of modules. If $S$ is stable, then $\underline{\pi}$ is the cokernel of $\underline{f}$, where $\pi:M\twoheadrightarrow M/f(S)$ is the canonical projection.
\end{prop}
\begin{proof}
Since $S$ is stable, the morphism $\underline{\pi}$ is a weak cokernel of $\underline{f}$. To show that the morphism $\underline{\pi}$ is an epimorphism in the stable category, we observe that $f(S)\subseteq Ker g$, for any homomorphism $M\rightarrow P$ with a projective $P$. Hence there is $s$ such that $s\pi=g$. Now it suffices to apply Theorem 5.5.
\end{proof}

\begin{cor}
Let $\Lambda$ be hereditary, left perfect and right coherent. For an arbitrary homomorphism  $f:K\rightarrow M$, the cokernel of the morphism $\underline{f}$ of the stable category is $\underline{\pi}$, where $\pi$ is the canonical projection $M\twoheadrightarrow M/f(R(K))$.
\end{cor}

\begin{proof}
By Theorem 5.2 the cokernel of $\underline{f}$ coincides with the cokernel of the morphism $\underline{f|_{R(K)}}$, where $f|_{R(K)}:R(K)\rightarrow M$ is the restriction of $f$ on $R(K)$. Now it suffices to apply Proposition 5.6.
\end{proof}
\vskip+2mm
Finally recall that, as it was noted in \cite{MZ}, the quotient functor \begin{center}
$Q:\Lambda-Mod\rightarrow \underline{\Lambda-Mod}$
\end{center} \noindent never preserves cokernels unless the stable category is trivial.

\vskip+2mm
\textit{Author's address:
Dali Zangurashvili, A. Razmadze Mathematical Institute of Tbilisi State University}
\textit{6 Tamarashvili Str., Tbilisi 0177, Georgia, e-mail: dali.zangurashvili@tsu.ge}


\begin{thebibliography}{99}
\bibitem{B} F. Borceux, Handbook of categorical algebra, 1, volume 50 of Encyclopedia of Mathematics and its Applications. Cambridge University Press, Cambridge, 1994. Basic category theory.

\bibitem{C} S. U. Chase, Direct products of modules, Trans. Amer. Math. Soc., 97: 457-473, 1960.

\bibitem{CR} R. R. Colby, E. A. Rutter, Generalizations of QF-3 algebras, Trans. Amer. Soc., 153(1971), 371-386.

\bibitem{H} A. Heller, The loop-space functor in homological algebra, Trans. Amer. Math. Soc, 96 (1960), 382-394.

\bibitem{J} J. P. Jans, Projective injective modules, Pacific J. Math., 9(4), 1103-1108(1959).


\bibitem{MZ} A. Martsinkovsky, D. Zangurashvili, The stable category of a left hereditary ring. J. Pure Appl. Algebra, 219(2015), 4061-4089.

\bibitem{R} J. J. Rotman, Advanced modern  algebra, Amer. Math. Soc., Prentice Hall, 2nd Ed., 2003.

\bibitem{Z1} D. Zangurashvili, Several constructions for factorization systems, Theory and Appl. Categ., 2004.

\bibitem{Z} D. Zangurashvili, Effective codescent morphisms, amalgamations and factorization systems, J. Pure and Appl. Algebra, 209(2007), 255-267.

\end{thebibliography}
\end{document}